\documentclass[11pt]{amsart}
\usepackage{geometry}                % See geometry.pdf to learn the layout options. There are lots.
\usepackage{graphicx}
\usepackage{amsmath}
\usepackage{amsfonts}
\usepackage{amssymb}
\usepackage{epstopdf}
\usepackage{xypic}
\xyoption{all}
\newcommand{\Hom}{\operatorname{Hom}}
\newcommand{\Mod}{\operatorname{Mod}}

\newcommand{\Out}{\operatorname{Out}}
\newcommand{\Aut}{\operatorname{Aut}}
\newcommand{\bZ}{\mathbb{Z}}

\newcommand{\bR}{\mathbb{R}}
\newcommand{\bC}{\mathbb{C}}

\newcommand{\bH}{\mathbb{H}}
\newcommand{\End}{\mathbf{End}}

\newcommand{\PGL}{\operatorname{PGL}}
\newcommand{\GL}{\operatorname{GL}}

\newcommand{\caSets}{\mathcal{SETS}}

\newcommand{\Sym}{\operatorname{Sym}}

\newcommand{\rOne}{\mathbf{1}}

\newcommand{\tr}{\operatorname{Tr}}

\newcommand{\Funct}{\mathbf{Funct}}
\newcommand{\dual}{^\vee}
\newtheorem{thm}{Theorem}
\newtheorem{lemma}[thm]{Lemma}

\newtheorem{prop}[thm]{Proposition}

\newtheorem{quest}[thm]{Question}

\numberwithin{thm}{section}

\title[Representations of Galois Groups]{Representations of Galois Groups on the Homology of Surfaces}
\author{Thomas Koberda and Aaron Michael Silberstein}
\email[T. Koberda]{koberda@math.harvard.edu}
\email[A. Silberstein]{asilbers@math.harvard.edu}
\address{Department of Mathematics, Science Center, Harvard University, 1 Oxford St., Cambridge, MA 02138}
\begin{document}
\begin{abstract}
Let $p:\Sigma'\to\Sigma$ be a finite Galois cover, possibly branched, with Galois group $G$.  We are interested in the structure of the cohomology of $\Sigma'$ as a module over $G$.  We treat the cases of branched and unbranched covers separately.  In the case of branched covers, we give a complete classification of possible module structures of $H_1(\Sigma',\bC)$.  In the unbranched case, we algebro-geometrically realize the representation of $G$ on holomorphic $1$-forms on $\Sigma'$ as functions on unions of $G$-torsors.  We also explicitly realize these representations for certain branched covers of hyperelliptic curves.  We give some applications to the study of pseudo-Anosov homeomorphisms of surfaces and representation theory of the mapping class group.
\end{abstract}
\maketitle
\begin{center}
\today
\end{center}
\tableofcontents
%\section{}
%\subsection{}

\section{Introduction}
Let $\Sigma$ and $\Sigma'$ be closed, oriented $2$-manifolds, and let $p: \Sigma'\rightarrow \Sigma$ be a Galois cover branched at a finite set of points.  Then the Galois group $G = \Aut(\Sigma'/\Sigma)$ acts on $\Sigma'$ and therefore on $H_1(\Sigma', R)$ and $H^1(\Sigma', R)$ where $R$ is a constant sheaf of abelian groups.  Let $\mathcal{R}_R$ be the representation on $H^1(\Sigma', R)$, and $\eta$ its character; since the homology of $\Sigma'$ is torsion-free, \[\mathcal{R}_R = \mathcal{R}_{\bZ} \otimes_{\bZ} R\] and the character $\eta$ of $\mathcal{R}_R$ is $\bZ$-valued and independent of $R$.

The first theorem we prove in section \ref{s:geometriccw} is
\begin{thm}[Geometric Chevalley Weil]
Let $\pi: \Sigma'\rightarrow \Sigma$ be an unramified, finite Galois cover of a complete complex algebraic curve $\Sigma$ of genus $g$ (that is, a compact, oriented, topological surface equipped with a complex structure) with Galois group $G$.  Then for a generic one-dimensional subspace
\[\mathbb{C}\omega \subseteq H^{1,0}(\Sigma),
\]
there exists a canonically determined set of points $\{p_1, p_2, \dots, p_{g-1}\} \in \Sigma$ and an isomorphism
\[
H^{1,0}(\Sigma')/(\pi^*\mathbb{C}\omega)\rightarrow \Funct(\bigcup_{i=1}^{g-1}\pi^{-1}(p_i),\mathbb{C}).
\]
given by restriction, which thus commutes with the $G$-action and is an isomorphism of $G$-representations.
\end{thm}
This implies the classical theorem of Chevalley and Weil \cite{CW} in the unramified case (and its natural refinement to take into account the Hodge decomposition using the Atiyah-Bott fixed point formula), but provides extra information about the representation which comes from the algebraic geometry of the situation.  It also has the advantage of working in characteristic $p$.

%The purpose of this paper is to understand the representation of $G$ on the homology of $\Sigma'$ in a geometric way.  In section \ref{s:geometriccw}, we use the geometry of the canonical map on a curve to realize, after making some choices, the regular representations which occur in the Chevalley-Weil theory for an unramified cover explicitly as functions on $G$-torsors.  We then show that these possible descriptions are indeed a very small subset of all possible descriptions of a product of regular representations as functions on torsors, and thus single out a number of such isomorphisms.

In section \ref{s:ramcw}, we consider general branched covers of $\Sigma$, and we classify so called topological representations of $G$, i.e. ones that occur as permutation representations on the ramification divisor.  In the case where $\Sigma$ is a disk, we will show:

\begin{thm}
Let $G$ be a finite group and $\rho$ a permutation representation.  Then $\rho$ is topological if and only if there is a surjective map $F\to G$ from a free group $F$ of finite rank and a cyclic subgroup $H$ that is the image of a generator of $F$ such that $\rho$ is the permutation representation associated to the $G$-set $G/H$.
\end{thm}

This result will allow us to prove, when $\Sigma$ is closed:
\begin{thm}
Let $\Sigma'\to\Sigma$ be a branched cover with branch points $\{q_i\}$, let $*_i\in p^{-1}(q_i)$ be a choice of points in the fiber of each branch point, and let \[V=\bigoplus_{i=1}^m V_i\] be the direct sum of the permutation representations of $G$ on $G/Stab(*_i)$.  Then, the $G$-module $H_1(\Sigma',\bC)$ is isomorphic to the $G$-module \[H_1(\Sigma'\setminus p^{-1}(B),\bC)/(V/\mathbf{1}),\] where $\mathbf{1}$ is the trivial representation of $G$.
\end{thm}

The subrepresentation $V$ above is a direct sum of topological representations.

We will discuss the issue of when the representation of $G$ on $H^{1,0}(\Sigma')$ is isomorphic to that on $H^{0,1}(\Sigma')$ in the case of branched covers.

We shall also discuss some applications of this theory to the study of the mapping class group in section \ref{s:MCG}.  Namely, we shall produce a class of representations of finite index subgroups of the mapping class group that arise from actions of the mapping class group on the homology of a characteristic cover.  Furthermore, we shall discuss the problem of finding the dilatation of a pseudo-Anosov mapping class, and prove:
\begin{thm}
Let $F^{u,s}$ be the unstable and stable foliations associated to a pseudo-Anosov mapping class $\psi$ with dilatation $K$.  $F^{u,s}$ remain nonorientable on every finite cover of $\Sigma$ if and only if the spectral radius of $\psi$'s action on the real homology on every finite cover is strictly less than $K$.
\end{thm}
In section \ref{s:cyclic} we give an analysis along the lines of \cite{Mc} of how the algebro-geometric description of the regular representations in section \ref{s:geometriccw} can be constructed in a situation with relatively mild branching, and propose some questions related to the geometric picture associated to a ramified cover for further research.
\section{Unramified Chevalley-Weil Theory}\label{s:geometriccw}
\subsection{The Geometric Unramified Chevalley-Weil Theory}
In this section, we choose a complex structure on our surface and give a proof of the unramified holomorphic Chevalley-Weil theorem which bypasses the use of the trace formula.  We will use throughout this section the tools in \cite{ACGH}.

Let $C$ be a smooth, algebraic curve of genus $g$ over $\bC$.  We begin with the following lemma:

\begin{lemma} Let $K$ be a generic canonical divisor on $C$.  Then we may write \[K = K_1 + K_2\] where \[K_1, K_2 > 0, \deg K_1 = \deg K_2 = g-1,\text{ and }h^0(K_1) = h^0(K_2) = 1.\]\end{lemma}
\begin{proof} By the definition of genus, $h^0(K) = g$ and by the theorem of Riemann-Roch, $\deg K = 2g - 2$.  Let $H_{\omega}$ be the hyperplane of $\mathbb{P}H^0(K)\dual$ given by $\omega$.  We have the canonical map \[
\kappa: C\rightarrow \mathbb{P}H^0(K)\dual.
\]
$H_\omega$ is a $g-2$-dimensional projective subspace of the target of $\kappa$, and contains $2g - 2$ points of $C$, which projectively generate $H_\omega$.  Consider a subset of $g-1$ points which projectively generates $H_\omega$ (that is, this subset is not contained in any proper linear subspace of $H_\omega$), and let $K_1$ be their sum.  Then $K_1$ and, by Riemann-Roch, $K - K_1 = K_2$ satisfy the hypotheses of the theorem.
\end{proof}
Let $i: K_2\rightarrow C$ be the inclusion map of the divisor $K_2$, considered as a closed subspace.  Then we have a short exact sequence of sheaves (where now we consider the divisors as a line bundle)
\begin{center}
\mbox{\xymatrix{
0 \ar[r] & K_1 \ar[r] & K \ar[r] & i_* i^* K_2 \ar[r] & 0.
}}
\end{center}
The third term is a skyscraper sheaf on the support of $K_2$ so by the long-exact cohomology sequence and Serre duality, we have an exact sequence
\begin{center}
\mbox{\xymatrix{
0 \ar[r] & H^0(K_1) \ar[r] & H^0(K) \ar[r]^\iota & H^0(i_*i^*K_2) \ar[r] & H^1(K_1) \ar[r]^\lambda & H^1(K).
}}
\end{center}
By a combination of the assumptions and Serre duality, $\lambda$ is the isomorphism
\[
\lambda: H^0(K_2)\dual\rightarrow H^0(\mathcal{O}_C)\dual.
\]
This gives that
\begin{center}
\mbox{\xymatrix{
0 \ar[r] & H^0(K_1) \ar[r] & H^0(K) \ar[r]^\iota & H^0(i_*i^*K_2) \ar[r] & 0
}}
\end{center}
is short exact.  $H^0(K_1)$ is generated by a single $\omega$, and thus
\[
H^0(i_*i^* K_2)\simeq H^0(K)/\bC\omega.
\]
But $H^0(i_*i^* K_2)$ can be (up to multiplication by $\bC^{*(g - 1)}$) given canonically the structure of $\Hom_{\caSets}(K_2, \mathbb{C})$, the complex-valued functions on the support of $K_2$.

Now, let $C'$ be an unramified, finite cover of $C$, and let $K = K_1 + K_2$ be a generic decomposition of the canonical divisor of $C$ as in the lemma; then
\[
K_{C'} = p^*(K) = p^*(K_1) + p^*(K_2)
\]
must be a decomposition of the canonical divisor of $C'$ as in the lemma.  To see this, consider the commutative diagram
\begin{center}
\mbox{\xymatrix{
& \Sym^{n(g-1)}C' \ar[dr]^{\eta_{C'}} & \\
\Sym^{g-1}C \ar[r]^{\eta_C} \ar[ur]^{\sum_{\gamma\in G} \gamma x} & \mathbb{P}H^0(C, K)\ar[r]^{\mathbb{P}p^*} & \mathbb{P}H^0(C', K)
}}
\end{center}
$\eta_C$ (and $\eta_{C'}$) defines a rational, nonzero map by taking $g - 1$ points of $C$ (and $n(g-1)$ points of $C'$, as the genus of $C'$ is $n(g-1)+1$) to the unique $g-2$-plane they span, when this makes sense; the nontrivial geometry in this argument is a standard theorem on the canonical map which says this map is regular and nonzero.  The map $\sum_{\gamma\in G} \gamma x$ takes an unordered $g-1$-tuple to an unordered $n(g-1)$-tuple by summing over all orbits.  The commutativity of this diagram shows that if $\eta_X$ is defined on $K_1$ then $\eta_{C'}$ is defined on $p^*(K_1)$ so $p^*(K_1)$ spans a plane; equivalently, $\eta_C$ and $\eta_{C'}$ are defined on a $g-1$- or $n(g-1)$-tuple considered as a divisor $D$ precisely when the divisors have $h^0(D) = 1$.  So now, we have, noncanonically,
\[
H^0(p^*(K_1))/\bC p^*\omega = \Hom_{\caSets}(K_2, \bC)
\]
which gives us, by considering the functions which are supported on a single fiber (modulo $\bC p^*\omega$, of course), a \textbf{canonical} decomposition into $(g-1)$ times the regular representation, given a trivial representation:
\begin{thm}
Let $\langle\cdot,\cdot\rangle$ be the inner product on $H^0(K_{\Sigma'})$ for a K\"ahler metric pulled back from one associated to the orientation class of $\Sigma$. For every $P \in p^*(K_2)$ there exists a form $\delta_P \in H^0(K_{\Sigma'})$ such that $\delta_P(P)\neq 0$ and $\delta_P(P') = 0$ for every $P' \in p^*(K_2), P' \notin P$ , such that $\langle P, P\rangle = 1$. This then gives an isomorphism compatible with $G$
\[
\mathbf{CW}: H^0(K_{\Sigma'})/\bC p^*\omega\rightarrow \Funct(p^*(K_2), \bC)
\]
which explicitly realizes the decomposition
\[
H^0(K_{\Sigma'}) = \rho_R^{g-1} \oplus \rOne.
\]
\end{thm}
\subsection{Explicitly Realizing Regular Representations}
Let $G$ be a finite group, and let $X$ be a $G$-torsor (that is, a set on which $G$ acts simply transitively; picking a basepoint identifies $X$ with $G$, but this basepoint is not of course canonical).  Then the regular representation of $G$ can be written as
\[
\rho_R = \mathbf{Ind}_e^G \mathbf{1} = \Funct(X, \mathbb{C}).
\]

There is a canonical, $G$-invariant Hermitian, positive-definite inner product on this representation:
\[
\langle \phi, \eta \rangle = \sum_{g\in G} \phi(gx)\overline{\eta(gx)}
\]
which is of course independent of the $x \in X$ chosen.  Let $x \in X$ and consider the function $\delta_x$ supported at $x$ and taking only the value $1$.  Then $\delta_x$ has the property that for all $g, g' \in G$ we have
\[
\langle g\delta_x, g'\delta_x\rangle = \left\{\begin{array}{ll} 1 & \text{ if } g = g' \\ 0 & \text{ if }g \neq g'\end{array}\right.\] 
\begin{prop}
Let $V, \langle\cdot, \cdot\rangle$ be a representation of $G$ in which we have chosen a $\delta$ with the property above.  Then this induces an injection
\[
\iota_x: \rho_R\rightarrow V
\]
which preserves the inner product, by
\[
\iota_x(\delta_x) = \delta.
\]
\end{prop}
We recall that $\rho_R$ has the property that any element $v$ of any other representation there is a unique $G$-map
\[
\varphi_v: \rho_R\rightarrow V
\]
given by
\[
\varphi(\delta_x) = v.
\]
and thus
\[
\Hom(\rho_R, V) \simeq V.
\]

Let $\delta'$ be another element of $\rho_R$ which shares this property.  Then the induced map $\varphi_{\delta'}$ preserves the inner product.  Thus, the set of vectors with this property is acted upon transitively by the subgroup of $\mathbb{C}[G]^*$ which preserves the inner product.  Since
\[
\End(\rho_R) = \mathbb{C}[G] = \prod_{V \,\mathrm{ irreducible}} M_{\dim V}
\]
with the product acting in such a way that the projectors of this algebra act as orthogonal projectors on $\rho_R$ with $\langle\cdot,\cdot\rangle$, the subgroup of $\mathbb{C}[G]^*$ which preserves $\langle\cdot,\cdot\rangle$ is 
\[
U(\rho_R) = \prod_{V\,\mathrm{ irreducible}} U(\dim V);
\]
The stabilizer is empty.  Furthermore, we can carry through the same analysis of
\[
\rho_R^i = \Funct(\coprod_i X, \mathbb{C}).
\]
Here, we want, instead of one $\delta$-function, a set of $i$ $\delta$-functions $\delta_{x_1}, \dots, \delta_{x_i}$, all translates of whom form an orthonormal basis.  Noting that
\[
\End(\rho_R^i) = M_n(\mathbb{C}[G]) = \prod_{V\,\mathrm{irreducible}} M_{i\dim V}
\]
Again noting that the projections on $\rho_R^i$ associated to idempotents are orthogonal projections, we have
\[
U(\rho_R^i) = \prod_{V\,\mathrm{irreducible}} U(i\dim V).
\]
Thus, we obtain the
\begin{thm}\label{thm:deltachoices} In $\rho_R^i$ there is an  $i^2|G|$-dimensional real manifold of choices of $\delta$-functions if we require our $\delta$-functions to be orthonormal with respect to a $G$-invariant product; otherwise, if we just require our $\delta$-functions to have the property that their translates form a basis, we have an $i^2|G|$-dimensional complex manifold of choices of $\delta$-functions.
\end{thm}
\subsection{Automorphisms in the Unramified Case}
Let \[p: \Sigma'\rightarrow \Sigma\] be an unramified Galois covering of a genus $g$ surface $\Sigma$ with Galois group $G$ and let $n = |G|$.  In the procedure for producing the unramified Chevalley-Weil representation, we make the following choices:
\begin{enumerate}
\item A complex structure $J$ on $\Sigma$, which induces a canonical bundle $\Omega^1_{\Sigma_g}$ of holomorphic differential forms and pulled back structures $p^*J$ and $\Omega^1_{\Sigma'_{p^*J}}$ on $\Sigma'$, and a unique hyperbolic metric of constant negative curvature on both $\Sigma$ and $\Sigma'$ (we can in fact use any metric here; another natural metric would be the metric pulled back from the Fubini-Study metric on the canonical embedding).
\item A pulled back differential form $\omega\in p^*H^0(\Omega^1_{\Sigma_J}) \subset H^0(\Omega^1_{\Sigma'_{p^*J}})$.
\item A set of $g-1$ points $P_1, \dots, P_g \in \Sigma'$ in general position in the hyperplane of $\mathbb{P}H^0(\Omega^1_{\Sigma'_{p^*J}})$ defined by the pulled back differential form.
\item Identifications $\alpha_i: p^{-1}(P_i) \rightarrow X$, where $X$ is a fixed $G$-torsor.
\end{enumerate} 
The geometric Chevalley-Weil theorem asserts that this induces an isomorphism
\begin{center}
\mbox{\xymatrix{
H^0(\Omega^1_{\Sigma'_{p^*J}})/\mathbb{C}\omega \ar[r]^{\iota} & \mathbf{Funct}(\bigcup_i p^{-1}(P_i), \mathbb{C})\ar[rr]^{\bigcup_i \alpha_i^*}
& & \mathbf{Funct}(X^i, \mathbb{C})
}}
\end{center}
which can be rewritten, using the Hodge inner product given by $g$, as
\[
\mathbf{CW}: H^0(\Omega^1_{\Sigma'_{p^*J}}) \rightarrow \mathbf{Funct}(X^i, \mathbb{C}) \oplus \mathbb{C}\omega.
\]
A choice of complex structure gives us a complex conjugation involution
\[
\overline{\cdot}: H^1(\Sigma, \mathbb{C}) \rightarrow H^1(\Sigma, \mathbb{C})
\]
If we write, according to the Hodge decomposition,
\[
H^1(\Sigma', \mathbb{C}) = H^0(\Omega^1_{\Sigma'_{p^*J}}) \oplus H^1(\mathcal{O}_{\Sigma'_{p^*J}})
\]
then $\overline{\cdot}$ switches the two direct summands, and we have an isomorphism
\[
\mathbf{CW} \oplus \overline{\mathbf{CW}}: H^1(\Sigma', \mathbb{C})\rightarrow \mathbf{Funct}(\coprod_i X, \mathbb{C}) \oplus  \mathbb{C} \oplus \mathbf{Funct}(\coprod_i X, \overline{\mathbb{C}}) \oplus \overline{\mathbb{C}},
\]
where the conjugate copies indicate that the functions naturally are related by the complex conjugation induced from the complex structure.

Contained inside $H^1(X, \mathbb{C})$ is the $\mathbb{Z}$-module (note that it is not a lattice, because it is not cocompact) $H^1(X, \mathbb{Z})$, which $G$ must preserve.  Furthermore, there is a unimodular intersection symplectic form $i(\cdot, \cdot)$ on $\mathbb{Z}$.  Hodge theory tells us that the imaginary part of the Hermitian product induced by $g$ on this lattice must coincide with this intersection pairing.

We call the data \[\langle J, \omega, \{P_i\}, \{\alpha_i\}\rangle\] \textbf{trivialization data}, and the identification of cohomology with a regular representation it generates \[\mathbf{CW}_{\langle J, \omega, \{P_i\}, \{\alpha_i\}\rangle} \oplus \overline{\mathbf{CW}}_{\langle J, \omega, \{P_i\}, \{\alpha_i\}\rangle}.\]

Note that our method of constructing $\delta$-functions is not necessarily compatible with the inner product we are using; indeed, we do not know if the $\delta$-forms are orthogonal to one another in the Hodge inner product.  However, If we fix our complex structure $J$, then we see that the trivialization data come in a $g$-dimensional complex family (the data of $\{P_i\}$ and $\{\alpha_i\}$ are discrete once we've picked a $J$ and $\omega$).  The sets $\{P_i\}$ generically vary continuously with the choice of $\omega$, and so it makes sense to identify the $\Funct(\coprod_i X, \mathbb{C})$ obtained by a choice of $\omega$ and $\{P_i\}$ with that obtained by a close $\omega'$ and $\{P_i'\}$.  It then makes sense to say that there is really a \textit{discrete} choice of $\delta$-functions, according with what $\{P_i\}$ and $\{\alpha_i\}$'s are chosen, as the $\omega$'s come in a connected family; however, the possible choices for $\delta$-functions form a $g^2n$-dimensional family, as we computed in \ref{thm:deltachoices}.  Thus, we may deduce that the possible ``geometric'' trivializations obtained in this way, fixing a $J$, do not come close to exhausting all possible trivializations of the representation and form an essentially discrete set of possible trivializations.

It would be interesting to try to use certain special geometric trivializations of $H^1(\Sigma', \mathbb{C})$ (for instance, coming from curves with Jacobians with large endomorphism rings) to obtain trivializations of $H^1(\Sigma', \mathbb{Z})$; it does not seem clear that the representation on $H^1(\Sigma', \mathbb{Z})$ can always be trivialized.  However, there is a unique $\mathbb{Q}$-form of any representation over $\mathbb{C}$ by Hilbert's Satz 90 (see \cite{G}) so $H^1(\Sigma', \mathbb{Q})$ can always be trivialized.  This leads us to the natural
\begin{quest}
What conditions make a geometric trivialization descend to $H^1(\Sigma', \mathbb{Q})$ or $H^1(\Sigma', \mathbb{Z})$?  Can one pick a trivialization datum such that this occurs?
\end{quest}

\section{The Ramified Chevalley-Weil Theory: Topological Theory}\label{s:ramcw}
Now, let $p:\Sigma'\to\Sigma$ be a finite branched Galois cover of degree $n$ of a closed surface of genus $g\geq 0$.  Let $G$ denote the group of deck transformations.  We wish to understand $H_1(\Sigma',\bC)$ as a $G$-module.  Suppose that $B=\{q_1,\ldots, q_m\}$ is the branch locus of the covering map, with ramification degrees $n_i$, $1\leq i\leq m$.  Evidently each $q_i$ lifts to precisely $n/n_i$ points, and a small loop around a particular lift of $q_i$ wraps around $q_i$ $n_i$ times under $p$.

A branched cover becomes an unbranched cover after removing the branching locus.  The fundamental group $\pi_1(\Sigma\setminus B)\cong F_{2g+m-1}$ admits a splitting into a free product $F_h*F_p$, where $F_h$ denotes the {\bf hyperbolic} part of the fundamental group and $F_p$ denotes the {\bf parabolic} part.  Precisely, this means that if we endow $\Sigma\setminus B$ with a hyperbolic structure, there is a splitting of the fundamental group into a free product as above, where $F_h$ is generated by the homotopy classes of certain geodesics on $\Sigma$, and $F_p$ is generated by certain loops around points in $B$.  The former consists entirely of hyperbolic elements in $Isom(\bH^2)$, and the latter is generated by parabolic elements.

We may select a smoothly embedded disk $D\subset \Sigma$ such that $B\subset D$.  As is standard, $\pi_1(\Sigma)$ admits a presentation \[\langle a_1,\ldots, a_g,b_1,\ldots b_g\mid\prod_{i=1}^g [a_i,b_i]=1\rangle,\] and this can be done in a way so that the homotopy class of $\partial D$ is represented by the relation.  $D\setminus B$ on the other hand is free on $m$ generators, given by small loops $\ell_i$ around each puncture, and the homotopy class of $\partial D$ represented by $\ell_1\cdots\ell_m$.  Van Kampen's theorem says that then \[\pi_1(\Sigma\setminus B)\cong\langle a_1,\ldots, a_g,b_1,\ldots b_g, \ell_1\ldots,\ell_m\mid\prod_{i=1}^g [a_i,b_i]=\ell_1\cdots\ell_m\rangle.\]

The cover $\Sigma'\setminus p^{-1}(B)$ is given by a homomorphism $\pi_1(\Sigma\setminus B)\to G$.  In order to describe $\Sigma'\setminus p^{-1}(B)$ geometrically, we break $\Sigma'\setminus B$ into $\Sigma\setminus D$ and $D\setminus B$.  This will allow us to consider $\Sigma'\setminus p^{-1}(B)$ as a union of covers of $\Sigma\setminus D$ and $D\setminus B$, glued together over lifts of $\partial D$.  The gluing condition forces the following immediate lemma:

\begin{lemma}\label{t:tech}
A homomorphism $\pi_1(\Sigma\setminus B)\to G$ is given by two homomorphisms, $\pi_1(\Sigma\setminus D)\to G$ and $\pi_1(D\setminus B)\to G$ subject to the condition that the images of the homotopy class $[\partial D]$ be equal.
\end{lemma}

Henceforth, assume that the pair $(\Sigma,B)$ is not $(S^2,point)$.  It is fairly straightforward to compute $H_1(\Sigma',\bC)$ and $H_1(\Sigma'\setminus p^{-1}(B),\bC)$ simply from the branching data, and $H_1(\Sigma'\setminus p^{-1}(B),\bC)$ is easy to understand as a $G$-module from unramified Chevalley-Weil theory.  The dimension of $H_1(\Sigma'\setminus p^{-1}(B),\bC)$ can be simply computed from the fact that $\pi_1(\Sigma\setminus B)$ is free of rank $2g+m-1$ and $|G|=n$: a standard Euler characteristic argument shows that $\pi_1(\Sigma'\setminus p^{-1}(B))$ must have rank $n(2g+m-2)+1$.  Furthermore, the Lefschetz fixed point theorem implies that as a $G$-module, $H_1(\Sigma'\setminus p^{-1}(B),\bC)$ is $2g+m-2$ copies of the regular representation together with one copy of the trivial representation.

To understand $H_1(\Sigma',\bC)$ as a vector space, we simply note that $p^{-1}(B)$ consists of \[k=\sum_{i=1}^m n/n_i\] points, so that filling them all in shows that the dimension of $H_1(\Sigma',\bC)$ is just $n(2g+m-2)+2-k$.  The effect of filling in the punctures on the $G$-module structure is simply to kill certain direct summands of the abstract $G$-module $H_1(\Sigma'\setminus p^{-1}(B),\bC)$.

Precisely, decompose $H_1(\Sigma'\setminus p^{-1}(B),\bC)$ as $\bigoplus_{\chi} d_i\cdot V_{\chi}$, where the subscript $\chi$ ranges over irreducible characters of $G$ and the $d_i$'s are the multiplicities with which they occur.  For each point $q\in B$, the $G$-orbit gives rise to a $G$-submodule of $H_1(\Sigma'\setminus p^{-1}(B),\bC)$ which is killed upon filling in $p^{-1}(q)$.

As a $G$-set, $p^{-1}(q)$ is just $G/Stab(*)$ for some $*\in p^{-1}(q)$.  Clearly, the associated $G$-submodule of $H_1(\Sigma'\setminus p^{-1}(B),\bC)$ induced by the $G$-set $p^{-1}(q)$ factors through a $G/H$-module, where \[H=\bigcap_{*\in p^{-1}(q)} Stab(*).\]  If $|B|\geq 2$ then we may assume that the map from the free vector space on $p^{-1}(q)$ into $H_1(\Sigma'\setminus p^{-1}(B),\bC)$ induced by puncturing is an injection, and thus gives a subrepresentation of $G$, whose character $\chi$ we can compute.  Let $g\in G/H$.  For each $*\in p^{-1}(q)$, $g$ stabilizes $*$ precisely if $g\in Stab(*)$.  We see that $p^{-1}(q)$, despite being a transitive $G/H$-set, breaks up further into orbits under the action of powers of $g$.  We have a map of $G/H$-sets \[G/H\to \prod_{*\in p^{-1}(q)} G/Stab(*),\] which takes an element $g$ to each of its cosets modulo $Stab(*)$.  For each point $*$, either $g\in Stab(*)$ or not.  If it is, then $*$ contributes a $1$ to $\chi(g)$.  If not, then $*$ is in a nontrivial cyclic orbit of $\langle g\rangle$.  Therefore, $\chi(g)$ is just the number of points $*\in p^{-1}(q)$ for which $g\in Stab(*)$.  In particular, it is equal to the number of singleton orbits for the $\langle g\rangle$-action.  By construction, $\chi(g)\leq |p^{-1}(q)|$, with equality only if $g$ is the identity.

If $|B|=1$, then the map from the free vector space on $p^{-1}(q)$ to $H_1(\Sigma'\setminus p^{-1}(B),\bC)$ induced by puncturing has a one-dimensional kernel.  In $G$, the image of the element $\prod_{i=1}^g [a_i,b_i]$ has some order, which is the ramification degree about the branch point.  The kernel is given by \[\bC\cdot\langle \sum_{*\in p^{-1}(q)} *\rangle.\]  On the level of representation theory, this is the same situation as the case when $|B|\geq 2$, simply with one trivial representation factored out.  In the general case, we patch up the punctures on $\Sigma'\setminus p^{-1}(B)$ by patching up the preimages of punctures in $\Sigma$ one at a time.  There is certainly an ambiguity as to in which order we patch up the punctures.  The ambiguity is really just a choice of trivial representation in the abstract description of $H_1(\Sigma'\setminus p^{-1}(B),\bC)$ as a $G$-module.

We wish to understand the $G$-action on the cover in more detail.  To this end, we have the following theorem:

\begin{thm}
Let $\Sigma'\to\Sigma$ be a branched cover as in the setup, let $*_i\in p^{-1}(q_i)$ be a choice of points in the fiber of each branch point, and let \[V=\bigoplus_{i=1}^m V_i\] be the direct sum of the permutation representations of $G$ on $G/Stab(*_i)$.  Then, the $G$-module $H_1(\Sigma',\bC)$ is isomorphic to the $G$-module \[H_1(\Sigma'\setminus p^{-1}(B),\bC)/(V/\mathbf{1}),\] where $\mathbf{1}$ is the trivial representation of $G$.
\end{thm}

\begin{proof}
 Let $X$ denote the $m$-times punctured disk.  A finite branched cover $p$ of the disk with branching locus equal to those $m$ points is given by a map $F_m\to G$, where $G$ is a finite group and the loops encircling each puncture are sent to nontrivial elements of $G$.  Fix a puncture $q$ and the homotopy class $\ell$ of a small loop going around $q$.  On the $G$-set $p^{-1}(q)$, the action of the image $h$ of $\ell$ in $G$ is clearly trivial.  Notice on the other hand that if $q$ is a puncture in $X$ whose image in $G$ is trivial, then there are precisely $|G|$ lifts of a little loop going around $q$, and the action of $G$ on these loops is just the regular $G$ action on itself.

If $|B|\geq 2$, the character of $G$ acting on the free vector space $V$ on $p^{-1}(q)$, which sits naturally inside of $H_1(\Sigma'\setminus p^{-1}(B),\bC)$, is just the permutation character of $G$ acting on $G/Stab(*)$ for any $*\in p^{-1}(q)$.  Clearly, this representation is a quotient of the regular representation.  If $|B|=1$, then we have a permutation character minus a trivial character, as explained above.  It is evident that the trivial representation appears as a direct summand of the permutation character, as all the values of the permutation character are nonzero integers, at least one of which is positive. 
\end{proof}

\subsubsection{Geometric permutation representations of finite groups}
The representation $V$ above is a direct sum of permutation representations associated to transitive $G$-sets.  A natural question is which permutation representation of $G$ can be realized geometrically.  In particular, if $*\in p^{-1}(q)$ for a puncture $q$ of the disk, which subgroups of $G$ can be realized as stabilizers of $*$, up to conjugacy?  We call permutation representations that arise in such a way {\bf topological}.

Certainly we cannot obtain all permutation representations this way.  For instance, let $Y$ be a once punctured disk.  If $G$ is a finite group, then the only surjective homomorphisms from $\pi_1(Y)$ to $G$ can be constructed when $G$ is cyclic.  In that case, we may write $G=\bZ/n\bZ$, and we may take a generator $y\in \pi_1(Y)$ to be sent to $1$.  Then, the puncture has exactly one lift and its stabilizer is the whole group.  In particular, the homology representation of $\bZ/n\bZ$ associated to this cover is trivial.

Note that a small neighborhood of a puncture $*$ on a multiply punctured surface $X$ can be taken to be homeomorphic to $Y$ as above.  We may assume that $Stab(*)$ acts by covering transformations on this neighborhood, and is therefore cyclic.  It follows that for any puncture $q$ on the base and $*\in p^{-1}(q)$, the group $Stab(*)$ is conjugate to the image of the homotopy class of a small loop $\ell$ encircling $q$.  We therefore have:

\begin{lemma}
Let $\rho$ be a topological permutation representation of a finite group $G$.  Then $G$ is the associated permutation representation on the $G$-set $G/H$, where $H$ is a cyclic subgroup of $G$.
\end{lemma}

Now let $X$ be a multiply punctured surface with free fundamental group $F$, let $G$ be a finite group and $\phi:F\to G$ a homomorphism.  The only relevant part of the group $G$ is the image of $\phi$.  Each element of $F$ has an image which generates a finite cyclic subgroup.  If a generator $x\in F$ is represented by a small loop about a puncture $q$, then $\phi(\langle x\rangle)$ will be the conjugacy class of $Stab(*)\in p^{-1}(q)$.

In order to analyze the topological permutation representations that can occur and avoid the technicalities introduced by lemma \ref{t:tech} we will assume for now that $X$ is just a multiply punctured disk with at least two punctures, and we take small loops about the punctures to be the generators of $\pi_1(X)$.  Let $p$ be any finite cover with Galois group $G$.  Gathering all the information we have accrued so far implies:

\begin{thm}
Let $G$ be a finite group and $\rho$ a permutation representation.  Then $\rho$ is topological if and only if there is a surjective map $F\to G$ from a free group $F$ of finite rank and a cyclic subgroup $H$ that is the image of a generator of $F$ such that $\rho$ is the permutation representation associated to the $G$-set $G/H$.
\end{thm}

In the general case when $X$ is a multiply punctured surface, the theorem above holds provided that the free group has generators chosen in such a way that small loops $\ell_i$ about the punctures are all generators (when the number of these punctures is greater than $1$, of course), and the condition on the image of the homomorphism $\phi$ to $G$ is that $\prod_i \phi(\ell_i)$ be expressible as a product of precisely $g$ commutators, where $g$ is the genus of $X$.

\subsubsection{Hodge theory and branched covers}
Recall that in the unramified case of Chevalley-Weil theory, the Atiyah-Bott fixed point theorem implies that the Galois group $G$ has isomorphic representations on $H^{1,0}(\Sigma')$ and on $H^{0,1}(\Sigma')$, after fixing a complex structure on $\Sigma$ that is then pulled back to $\Sigma'$.  These two representations are a priori dual to each other, and this is all that we can get in the case of a branched cover.

Recall that the Atiyah-Bott fixed point theorem says that if $g$ is an automorphism of a Riemann surface, then \[\tr(g\mid H^{0,1})=1-\sum_{g\cdot z=z} (1-g'(z))^{-1}.\]  As before, let $D\subset\Sigma$ be a smoothly embedded disk that contains the ramification points of the cover.  If we puncture at the ramification points and choose a homomorphism from the fundamental group to $G$, the only condition is that $[\partial D]$ has trivial image in $G$.

\begin{prop}
Let $p:\Sigma'\to\Sigma$ be a finite abelian Galois cover branched over at most two points.  Then the $G$-modules $H^{1,0}(\Sigma')$ and $H^{0,1}(\Sigma')$ are isomorphic.
\end{prop}
\begin{proof}
For covers branched over at most one point there is nothing more to show.  Suppose that $p$ is branched over exactly two points.  Let $\alpha_1$ and $\alpha_2$ be small loops around the ramification points.  Then the images of the homotopy classes of these two loops are inverse to each other.  The derivative of a deck transformation fixing a lift of one of the ramification points is a root of unity.  Any such deck transformation fixes any lift of the other ramification point, and the derivative must also be the conjugate root of unity.  Plugging this fact into the Atiyah-Bott fixed point theorem shows that the character of any deck transformation is conjugation invariant.
\end{proof}

When $G$ is nonabelian we no longer have the condition that $[\partial D]$ has trivial image, so it is not difficult to find branched coverings where the holomorphic and anti-holomorphic parts of the homology are only dual to each other.  If more ramification points are allowed, even cyclic covers can be made to have nonisomorphic holomorphic and anti-holomorphic homology representations.  An easy example is a degree $11$ cyclic branched cover of a torus, branched over $3$ points, where we send loops about the punctures to $2$, $4$ and $5\pmod{11}$.  In this case, each ramification point has exactly one lift of $\Sigma'$, so that any generator of $\bZ/11\bZ$ fixes each lift.  An easy computation shows that the generator $1$ has a character which is not real.

\subsubsection{The dilatation of a pseudo-Anosov homeomorphism}
The main reference for well-known unjustified statements in this subsection is the volume \cite{FLP}.  It is well-known from the Thurston theory of surface homeomorphisms that if $\psi\in Mod(\Sigma)$ is a pseudo-Anosov map then there is an invariant pair of transverse measured foliations on $\Sigma$, $F^s$ and $F^u$, which are preserved by some representative in the isotopy class of $\psi$, which we will not distinguish from $\psi$ notationally.  Together, these measured foliations provide natural coordinates on $\Sigma$ that patch together to give a projective class of holomorphic quadratic differentials $\bR^*\cdot f(z)\, dz^2$ on $\Sigma$ in such a way that the foliations become the horizontal and vertical directions for $\bR^*\cdot f(z)\, dz^2$.  Fixing one such nonzero differential $\eta$, the foliations will not be orientable unless $\eta$ can be globally expressed as the square of an abelian differential $\omega$.  Taking the de Rham cohomology class $[\omega]$ of $\omega$, $[\omega]$ becomes a unique (up to scale) eigenvector for the action of $\psi$ on the cohomology with eigenvalue $K$, the dilatation of $\psi$.  It is well known (from the work of the first author, for instance) that no larger eigenvalue is possible for the action of $\psi$ on the real cohomology of $\Sigma$.

In general, $\eta$ will only locally be the square of an abelian differential.  The spectral radius of the map of $\psi$ on real cohomology will have spectral radius strictly smaller than $K$.  In precise terms:

\begin{thm}
Let $F^{u,s}$ be the unstable and stable foliations associated to a pseudo-Anosov mapping class $\psi$ with dilatation $K$.  $F^{u,s}$ remain nonorientable on every finite cover of $\Sigma$ if and only if the spectral radius of $\psi$'s action on the homology on every finite cover is strictly less than $K$.
\end{thm}
\begin{proof}
Suppose that the foliations become orientable on a finite cover.  Then, there is a canonical real homology class associated to $F^u$ and $F^s$ called their Ruelle-Sullivan classes (see \cite{RS} for definitions and properties).  These classes will be eigenvectors for the action of $\psi$ on the real homology with eigenvalues $K^{-1}$ and $K$, respectively.  The classes themselves can be realized as a train tracks associated to the foliations.  See \cite{PH} for more details on this construction.

Conversely, we may just suppose that $F=F^s$ remains nonorientable on every finite cover of $\Sigma$.  In this case, $F$ is not locally orientable.  Branching over the singularities of $F$, we may obtain a degree two branched cover of $\Sigma$ to which $F$ lifts to an orientable foliation.

A general property of the stable measured foliation $F$ is that if we consider the image of a simple closed curve $c$ under iteration by $\psi$, the images converge to the stable lamination $\Lambda$ of $\psi$, and the measure is obtained by appropriately rescaling the geometric intersection number.  The measure can be integrated to obtain $F$.

Let $c$ be any simple closed geodesic curve on $\Sigma$ that is transverse to $F$ and let $i_G$ denote geometric intersection number.  General theory says that \[i_G(c,\psi(F))=K\cdot i_G(c,F).\]  On the other hand, let $i_A$ denote algebraic intersection number, and identify the weighted train track $\tau$ associated to $F$ with the Ruelle-Sullivan class of $F$.  Then, \[i_A(c,\psi(\tau))=K\cdot i_A(c,\tau).\]  This implies that if the homology class of $[c]$ is nonzero and has nonzero algebraic intersection with $\tau$, the the projective class of $[\psi^n(c)]$ will converge to that of $\tau$ as $n$ goes to infinity.

Certainly there can be simple closed curves that are homologically nontrivial, have nonzero geometric intersection with $F$ but have trivial algebraic intersection with $F$.  In this case, we cannot deduce convergence in homology to $\tau$ from geometric convergence.  It is known that for an orientable invariant foliation for a pseudo-Anosov map the dilatation appears as a simple eigenvalue of the induced map on homology by the work of Thurston and McMullen, so that we may restrict our attention to homologically nontrivial curves with nonzero algebraic intersection with $\tau$.

Therefore, there is a cone of positive vectors in $H_1(\Sigma,\bR)$ which is preserved and is mapped strictly into itself by the linear map induced by $\psi$.  It follows that the induced matrix on the real homology is Perron-Frobenius, and therefore has a unique largest eigenvalue which is real and positive.

Notice that if $p:\Sigma'\to\Sigma$ is a (possibly branched) cover, then there is a canonical injection $H_1(\Sigma,\bR)\to H_1(\Sigma',\bR)$ given by the transfer homomorphism.  Furthermore, there is a splitting $H_1(\Sigma',\bR)\cong H_1(\Sigma,\bR)\oplus V$, where $V$ is the kernel of the map $p_*$ induced my the cover.  The homology class of $\tau$ lies in $V$, since we may assume $F$ is nonorientable on the base.  Also, since the attracting positive ray in $H_1(\Sigma',\bR)$ determined by $\tau$ is unique, there are no other eigenvectors with eigenvalue $K$.  It follows that the spectral radius of the map on $H_1(\Sigma,\bR)$ induced by $\psi$ is strictly less than $K$.
\end{proof}

As is mentioned in the proof, the invariant foliations for $\psi$ can be encoded in a combinatorial object, namely a measured train track $\tau$.  This is a weighted embedded graph satisfying some local smoothness conditions, a measure compatibility condition at the switches, and a geometric condition on the complementary regions.  The graph $\tau$ will be orientable in a way compatible with the smoothness criteria if and only if the invariant foliations are orientable.  It is generally not possible to pass to a finite cover of $\Sigma$ and obtain an orientable train track, as the obstructions to orientability are local.  However, a double branched cover $\Sigma'$ of $\Sigma$ with branching to degree two over each over one point in each of the (unpunctured) complementary regions furnishes a surface together with an orientable train track and a lift of $\psi$ that preserves $\tau$ and scales the measure by a factor of exactly $K$.

It follows that there is a special subspace of $H_1(\Sigma',\bC)$ that contains the missing data necessary to reconstruct the dilatation of $\psi$ and which is missing from the homology data of a usual Galois unbranched cover.  This data could be called the {\bf branched homological data}, as opposed to the unbranched data which can be seen entirely on the level of the fundamental group.

We follow the paradigm outlined above.  The branching locus is $B$ and we decompose $\Sigma\setminus B$ into $\Sigma\setminus D$ and $D\setminus B$ for some smoothly embedded disk containing the branching locus.  The target group is $\bZ/2\bZ$ which is abelian, so that the image of the homotopy class of $\partial D$ is forced to be trivial.  The branching occurs over an even number of punctures by the condition that the homotopy class of $\partial D$ has trivial image.  If $q$ is such a branch point, $p^{-1}(q)$ consists of one point and its stabilizer is all of $\bZ/2\bZ$.  Therefore as one would expect, the representations in $H_1(\Sigma'\setminus p^{-1}(B),\bC)$ that are killed by filling in the branch points are all trivial.

Now suppose that $\Sigma$ was closed to begin with, so that $|B|=n$ is even and positive.  The $\bZ/2\bZ$-module $H_1(\Sigma'\setminus p^{-1}(B),\bC)$ consists of $2g+n-1$ copies of the regular representation, together with one copy of the trivial representation.  Filling in the $n$ branch points kills $n-1$ copies of the trivial representation.  Therefore, if $\rho_z$ denotes the regular representation $\rho_R$ of a group with the trivial representation $\mathbf{1}$ factored out, the resulting module $H_1(\Sigma'\setminus p^{-1}(B),\bC)$ is \[2g\rho_R+(n-1)\rho_z+\mathbf{1}.\]  Comparing this module to a standard degree two unbranched cover of $\Sigma$, we would obtain $(2g-2)\rho_R+2\mathbf{1}$.

Thus we obtain a precise result about the missing data needed to obtain the dilatation of $\psi$:

\begin{thm}
The branched homological data needed to detect the dilatation of a mapping class $\psi$ is contained in one copy of $\rho_R$, together with $n$ copies of $\rho_z$.
\end{thm}

It is important not to be mislead by this theorem, as the splitting of $H_1(\Sigma',\bC)$ is not canonical, in the sense that there is no map \[(2g-2)\rho_R+2\mathbf{1}\to 2g\rho_R+(n-1)\rho_z+\mathbf{1}\] which is compatible with the action of $\psi$.

\section{Representations of finite index subgroups of $\Mod(\Sigma)$ and $\Mod^1(\Sigma)$}\label{s:MCG}
Let $p:\Sigma'\to\Sigma$ be a finite Galois cover with Galois group $G$.  Assume furthermore that it is characteristic so that $\pi_1(\Sigma')$, viewed as a subgroup of $\pi_1(\Sigma)$, is invariant under the automorphisms of $\pi_1(\Sigma)$.  It is standard that we may identify the mapping class group of $\Sigma$, namely $\Mod(\Sigma)$, with a subgroup of index $2$ in $\Out(\pi_1(\Sigma))$, and the mapping class group with one marked point $\Mod^1(\Sigma)$ with a subgroup of index $2$ of $\Aut(\pi_1(\Sigma))$.

There is a map $\Mod(\Sigma)\to Sp_{2g}(\bZ)$ given by looking at the action on the homology of $\Sigma$.  There is generally no map $\Mod(\Sigma)\to Sp_{2g'}(\bZ)$ given by an action on the homology of $\Sigma'$, since there is an ambiguity as to which lift of a particular mapping class is to be chosen, and because the Galois group of the cover acts nontrivially on the homology of $\Sigma'$.  We do get a map $\Mod^1(\Sigma)\to Sp_{2g'}(\bZ)$, however.  We remark that it is known that each mapping class acts nontrivially on the homology of some finite cover of $\Sigma$, and that such a cover can be taken to be nilpotent (see \cite{K}).

Consider the $G$-module $H_1(\Sigma',\bC)$  We may decompose it into irreducible modules.  Let $V$ be a particular irreducible $G$-module and $\psi$ a mapping class.  Then we obtain a new irreducible $G$-module $V_{\psi}$.  As a vector space, $V_{\psi}$ is just $\psi(V)$.  The $G$-action is given by $\psi(g\cdot V)=\psi(g)\cdot V_{\psi}$.  Thus, the mapping class group permutes the isotypic representations of $H_1(\Sigma',\bC)$.  Since there are only finitely many of these, we find a finite index subgroup $S_{\Sigma'}$ of $\Mod^1(\Sigma)$ that fixes the isotypic components of the $G$-module $H_1(\Sigma',\bC)$.  There is a further finite index subgroup $S_{\Sigma',G}$ that commutes with the $G$-action on the homology, which maps to $Sp_{2g'}^G(\bZ)$, the group of symplectic matrices commuting with the $G$-action.

In general, it seems difficult to give a precise characterization of the image of this map.  In the case of an abelian cover, the $G$-action on the irreducible representations is given by scalar multiplication.  It follows that a finite index subgroup of $\Mod(\Sigma)$ acts on the projectivized isotypic pieces of $H_1(\Sigma',\bC)$.  This representation has been studied by Looijenga in \cite{L}.  He shows that when the genus is at least $3$, the image of this map has finite index.  This is in striking contrast to the case of braid groups for instance, where the image may not be arithmetic (see \cite{Mc}).
\section{Cyclic Covers of Hyperelliptic Curves Branched at Two Points, McMullen's Approach, and the ``Regular'' Part of a Ramified Representation}\label{s:cyclic}
In \cite{Mc}, McMullen analyzes cohomology representations for cyclic, ramified covers of $\mathbb{P}^1$ by using the large automorphism group of $\mathbb{P}^1$.  He writes down a model for a branched cover in which the branch points of the base are acted upon simply transitively by a cyclic subgroup of $PGL_2(\mathbb{C})$, then uses this information to write down a cyclic cover whose equations he can write down explicitly.  With this information, he can write down a basis for $H^0(K)$ of the cover with a very explicit Galois action.

Since a genus $g\geq 2$ surface (regardless of the number of punctures) has an automorphism group of cardinality $\leq 84(g-1)$, we cannot expect to find such a general method in the high-genus case for describing such covers.  There is, though, one class of curves to which we can apply this method: hyperelliptic curves.

In this section, we will often write down equations for projective varieties using affine coordinates. A hyperelliptic curve is a curve which admits a degree $2$ map to $\mathbb{P}^1$.  We collect the following standard, equivalent formulations of hyperellipticity (see \cite{ACGH}):
\begin{thm}
Let $X$ be a smooth curve of genus $g \geq 2$.  Then $X$ is hyperelliptic if and only if the following equivalent criteria hold:
\begin{enumerate}
\item There exists a map $p: X\rightarrow \mathbb{P}^1$ of degree $2$.
\item The canonical map $\kappa: X\rightarrow \mathbb{P}^{g-1}$ is generically $2$-to-$1$.
\item The canonical map $\kappa: X\rightarrow \mathbb{P}^{g-1}$ is not injective.
\item $X$ is isomorphic to a plane curve by the (affine) equation $y^2 = f(x)$ where $f(x)$ is an equation of degree $2g+2$.
\end{enumerate}
The Weierstrass points of $X$ are then $(x, 0)$ where $f(x) = 0$, and the weight sequence at those points is $0, 2, \dots, 2g-2$.  Furthermore, we can write down a basis for differentials on $X$ in terms of its expression as a plane curve:
\[
H^0(K_X) = \left\{\frac{dx}{y}, x\frac{dx}{y}, \dots, x^{g-1}\frac{dx}{y}\right\}.
\]
\end{thm}

Let now $X$ be a hyperelliptic curve given by the equation
\[
y^2 = \sum_{i=0}^{2g+2} a_i x^i,
\]
where $a_0 \neq 0$ and $n = 2g + 2$.  The genus of $X$ is $g$.  Then if we define $X'$ by the equation
\[
y^2 = \sum_{i=0}^{2g+2} a_i x'^{mi}
\]
then $X'$ is a cyclic cover of $X$ with group $G = \mathbb{Z}/m\mathbb{Z}$ fully ramified at $(0, \pm \sqrt{a_0})$ with covering map given by
\[
(y,x') \mapsto (y,x'^m)
\]
and isomorphism $G \rightarrow \mu_m$
\[
g \mapsto (x' \mapsto \zeta x')
\]
where $g$ is a fixed generator of $G$ and $\zeta$ is a primitive root of unity.

We then have that a basis of differential forms for $X'$ is given by
\[
\left\{\frac{dx'}{y}, x'\frac{dx'}{y}, x'^2 \frac{dx'}{y}, \dots, x'^{mg-1}\frac{dx'}{y}\right\}.
\]
The $1$-dimensional vector space spanned by each of these basis elements forms a one-dimensional representation of $G$:
\[
g(x'^\ell \frac{dx'}{y}) = \zeta^{\ell+1}x'^\ell\frac{dx'}{y}.
\]
This gives us an explicit realization of the representation
\[
H^0(K, X') = \rho_R^g
\]
as each character of $G$ appears with multiplicity $g$ (as representations occur with multiplicity one in the regular representations of abelian groups).

When we analyze the Chevalley-Weil representations by the fixed point perspective, there is a contribution from the branch points, and a contribution which comes from the ``unramified'' part of the cover.  In that case we can ``localize away'' this contribution to the branch points by analyzing $H^0(K-B)$ and $H^0(K)/H^0(K-B)$ separately, where
\[
B = (n-1)(0, \sqrt{a_0}) + (n-1)(0,-\sqrt{a_0})
\]
is the ramification divisor on $X'$.  We can write down $H^0(K-B)$; this is the vector space spanned by
\[
\left\{x'^{m-1}\frac{dx'}{y}, x'^m\frac{dx'}{y}, x'^{m+1} \frac{dx'}{y}, \dots, x'^{mg-1}\frac{dx'}{y}\right\}.
\]
This tells us that
\[
h^0(K-B) = m(g-1) + 1.
\]
The induced map
\[
X'\rightarrow \mathbb{P}H^0(K-B)\dual
\]
then factors through the hyperelliptic quotient, is basepoint free, and thus has image a rational normal curve of degree $m(g-1)$; furthermore, 
\[
\deg K - B = 2m(g-1).
\]
Let $\omega$ be a form on $X$ and let $p^*\omega$ be its pullback; this fixes a trivial representation of $G$.  Then in the notation of the section on the unramified Chevalley-Weil theory, $H_{p^*\omega} \cap X'$ is a disjoint union of $g-1$ $G$-torsors and we may realize by exactly the same method
\[
H^0(K-B) \simeq \rho_R^{g-1} \oplus \mathbf{1}.
\]
Indeed, the same argument, combined with Clifford's theorem and standard results on rational normal curves \cite{ACGH}, implies that 
\begin{thm}
Let $p: \Sigma'\rightarrow \Sigma$ be a finite, Galois cover of Riemann surfaces with Galois group $G$, let $B$ be the ramification divisor of $p$, and let $J$ be a complex structure on $\Sigma$ such that on $\Sigma'$
\[
h^0(K) - h^0(K-B) = \frac{1}{2}\deg B.
\]
Then
\[
h^0(K-B) \simeq \rho_R^{(g-1)} \oplus \mathbf{1},
\]
and the isomorphism is geometrically realizable.
\end{thm}
We thus ask the question
\begin{quest}
For some  $J$ on $\Sigma$, is it true that on $\Sigma'$ with complex structure given by $p^*J$ \[h^0(K) - h^0(K-B) = \frac{1}{2}\deg B.\]
\end{quest}
If this were true, then by Clifford's theorem (see \cite{ACGH}), it would imply a positive answer to the following weaker question:
\begin{quest}\label{quest:hyperell}
Is there some $J$ on $\Sigma$ such that $\Sigma'$ with complex structure given by $p^*J$ is hyperelliptic?
\end{quest}
The answer to this question is negative.  We now describe the process for constructing counterexamples.
\begin{lemma}
Let $X$ be a hyperelliptic curve, and $\varphi: X\rightarrow \mathbb{P}^1$ its degree $2$ map to $\mathbb{P}^1$, $G$ the Galois group of this cover, consisting of the identity element and the hyperelliptic involution, and $B\subset \mathbb{P}^1$ the branch points of this cover.  Then we may write
\[
1 \rightarrow G \rightarrow \Aut(X)\rightarrow P_B\rightarrow 1
\]
where $P_B$ is the group of all automorphisms of $\mathbb{P}^1$ which permute $B$ and $G$ is central.
\end{lemma}
\begin{proof}
We have a map $\Aut(X)\rightarrow P_B$ by restricting to the action on branch points; here, we use the fact that an automorphism of a curve must always permute the Weierstrass points, and the Weierstrass points of a hyperelliptic curve are the branch points of $\varphi$.  Let now $\alpha \in P_B$.  Then we have the diagram
\begin{center}
\mbox{\xymatrix{
X\ar[d]^{\varphi} \ar@{-->}[r]^{\tilde{\alpha}} & X\ar[d]^{\varphi} \\
\mathbb{P}^1 \ar[r]^{\alpha} & \mathbb{P}^1.
}}
\end{center}
where the dotted arrow $\tilde{\alpha}$ exists because, once we pass to function fields, $\varphi$ and $\varphi\circ\alpha$ give splitting fields of the same polynomial, and so are noncanonically isomorphic; this lifts $P_B$ to $\Aut(X)$.  To show that $G$ is the entire kernel of the projection to $P_B$, we note that any element of $\Aut(X)$ induces an automorphism of $X/G$ because automorphisms act on the image of the canonical map.  If an element of $\Aut(X)$ fixes $2g+2$ points in the image of the canonical map, it must act by the identity on the image $\mathbb{P}^1$ and hence must be in $G$.
\end{proof}
This lemma limits the possible automorphism groups of a hyperelliptic curve, so to answer Question \ref{quest:hyperell} in the negative, we need to find a covering of curves whose Galois group cannot be written as a subgroup of $\PGL_2(\bC)$, possibly centrally extended by $\bZ/2$.  We may use the fact that $\GL_2(\bC)$ acts faithfully on a vector space to realize that we cannot have a finite abelian subgroup of arbitrarily high rank (for instance, six will do).
\begin{thm}
Let $\pi: \Sigma'\rightarrow \Sigma$ be a cover (branched or not) of Riemann surfaces with Galois group finite and abelian of rank $\geq 6$.  Then there is no complex structure on $\Sigma$ such that the pullback to $\Sigma'$ is hyperelliptic.  In particular, we cannot localize the contributions to the Chevalley-Weil formula at the branch points for such covers.
\end{thm}
\section{Acknowledgements}
The authors would like to thank the numerous people whose insights and comments have been invaluable.  The authors thank Mladen Bestvina, Bill Goldman, Mark Goresky, Joe Harris, Jack Huizenga, Samuel Isaacson, Dan Margalit, Andy Putman, Peter Sarnak, and Juan Souto for helpful discussions, and Tom Church and Jordan Ellenberg for comments on a draft of this paper.  The authors especially thank their advisers Benedict Gross and Curt McMullen.  The first author is supported by a National Science Foundation Graduate Research Fellowship, and the second author is supported by an NDSEG fellowship.

\end{document}